\documentclass{amsart}
\usepackage[margin=1in]{geometry}
\usepackage{verbatim}
\usepackage[textsize=scriptsize]{todonotes}
\usepackage{stmaryrd}

\usepackage{amsmath}
\usepackage{amssymb}
\usepackage{amsthm}
\usepackage{amscd}
\usepackage{enumerate}
\usepackage[pdfusetitle,unicode,hidelinks]{hyperref}
\usepackage{bbm}
\usepackage{etoolbox}

\usepackage[utf8]{inputenc}

\newcommand{\tomemail}{\href{mailto:tom.bachmann@zoho.com}{tom.bachmann@zoho.com}}

\newtheorem{proposition}{Proposition}
\newtheorem{corollary}[proposition]{Corollary}
\newtheorem{lemma}[proposition]{Lemma}
\newtheorem{theorem}[proposition]{Theorem}

\newtheorem*{theorem*}{Theorem}
\newtheorem*{corollary*}{Corollary}
\newtheorem*{proposition*}{Proposition}
\newtheorem*{lemma*}{Lemma}
\theoremstyle{definition}
\newtheorem{definition}[proposition]{Definition}

\newtheorem*{definition*}{Definition}
\newtheorem*{construction*}{Construction}
\theoremstyle{remark}
\newtheorem{remark}[proposition]{Remark}

\newtheorem{example}[proposition]{Example}

\newcommand{\id}{\operatorname{id}}
\newcommand{\Z}{\mathbb{Z}}
\def\C{\mathbb C}

\let\scr=\mathcal
\let\bb=\mathbb
\newcommand{\Gm}{{\mathbb{G}_m}}

\def\A{\bb A}
\def\P{\bb P}

\newcommand{\DM}{\mathcal{DM}}
\newcommand{\SH}{\mathcal{SH}}

\DeclareMathOperator*{\colim}{colim}
\DeclareMathOperator*{\hocolim}{hocolim}
\let\lim=\relax
\DeclareMathOperator*{\lim}{lim}

\def\PSh{\mathcal{P}}

\def\Spc{\mathcal{S}\mathrm{pc}{}}

\newcommand{\iso}{\cong}
\newcommand{\wequi}{\simeq}

\def\adj{\leftrightarrows}

\DeclareRobustCommand{\ul}{\underline}

\let\cat=\mathrm
\def\Sm{{\cat{S}\mathrm{m}}}
\def\Sch{\cat{S}\mathrm{ch}{}}

\def\Zar{\mathrm{Zar}}

\def\mot{\mathrm{mot}}
\newcommand{\et}{{\acute{e}t}}

\def\pt{*}
\def\naive{\mathrm{naive}}
\def\mot{\mathrm{mot}}
\def\aff{\mathrm{aff}}
\def\Aff{\mathrm{Aff}}
\def\Ft{\mathrm{Ft}}

\def\sPre{\mathrm{sPre}}
\def\sShv{\mathrm{sShv}}
\def\Pre{\mathrm{Pre}}
\def\G{\bb G}
\def\fpqc{\mathrm{fpqc}}
\def\et{\mathrm{\acute et}}
\def\Zar{\mathrm{Zar}}
\newcommand{\Sing}{\operatorname{Sing}}
\newcommand{\fpsr}[1]{\llbracket #1 \rrbracket}
\newcommand{\flsr}[1]{(\!( #1 )\!)}
\newcommand{\Gr}[1]{\mathrm{Gr}_{#1}}
\newcommand{\infinite}{infinite\textsuperscript{\ref{fn:infinite}}}

\title{Affine Grassmannians in $\A^1$-algebraic topology}
\date{\today}

\author{Tom Bachmann}
\address{Fakultät Mathematik, Universität Duisburg-Essen, Essen, Germany}
\email{\tomemail}

\begin{document}

\maketitle

\begin{abstract}
Let $k$ be a field. Denote by $\Spc(k)_\pt$ the unstable, pointed motivic homotopy
category and by $R^{\A^1} \Omega_\Gm: \Spc(k)_\pt \to \Spc(k)_\pt$ the
($\A^1$-derived) $\Gm$-loops
functor. For a $k$-group $G$, denote by $\Gr{G}$ the affine Grassmannian of $G$.
If $G$ is isotropic reductive, we provide a canonical motivic equivalence $R^{\A^1} \Omega_\Gm G \wequi
\Gr{G}$. We use this to compute the motive $M(R^{\A^1} \Omega_\Gm G) \in \DM(k,
\Z[1/e])$.
\end{abstract}


\section{Introduction}

This note deals with the subject of $\A^1$-algebraic topology. In other words
it deals with with the $\infty$-category $\Spc(k)$ of motivic spaces over a base
field $k$, together with
the canonical functor $\Sm_k \to \Spc(k)$ and, importantly, convenient
\emph{models} for $\Spc(k)$. Since our results depend crucially on the seminal
papers \cite{asok2015affine-I, asok2015affine}, we shall use their
definition of $\Spc(k)$ (which is of course equivalent to the other definitions
in the literature): start with the category $\Sm_k$ of smooth (separated, finite
type) $k$-schemes, form the universal homotopy theory on $\Sm_k$ (i.e. pass to the
$\infty$-category $\PSh(\Sm_k)$ of space-valued presheaves on $k$), and then
impose the \emph{relations} of Nisnevich descent and contractibility of the
affine line $\A^1_k$ (i.e. localise $\PSh(\Sm_k)$ at an appropriate family of
maps).

The $\infty$-category $\Spc(k)$ is presentable, so in particular has finite
products, and the functor $\Sm_k \to \Spc(k)$ preserves finite products. Let $*
\in \Spc(k)$ denote the final object (corresponding to the final $k$-scheme
$k$); then we can form the pointed unstable motivic homotopy category
$\Spc(k)_\pt := \Spc(k)_{/*}$. It carries a symmetric monoidal structure coming
from the smash product. Thus, for any $P \in \Spc(k)_\pt$ we have the functor
$P \wedge \bullet: \Spc(k)_\pt \to \Spc(k)_\pt$. By abstract nonsense, this
functor has a right adjoint $\Omega_P: \Spc(k)_\pt \to \Spc(k)_\pt$, called the
($\A^1$-derived) $P$-loops functor.

For us, the most important instance of this is when $P = \Gm$ corresponds to the
pointed scheme $\Gm := (\A^1 \setminus 0, 1) \in \Sm_k$. Indeed studying the functor
$\Omega_\Gm$ is one of the central open problems of unstable motivic homotopy theory,
since it is crucial in the passage from unstable to stable motivic homotopy
theory. (The functor $\Omega_{S^1}$ is similarly important but much better
understood.) The main contribution of this note is the computation of
$\Omega_\Gm G$, where $G$ is (the image in $\Spc(k)_\pt$ of)
an appropriate group scheme, as corresponding via
the functor $(\Sm_k)_\pt \to \Spc(k)_\pt$ to a certain ind-variety known as the
\emph{affine Grassmannian} $\Gr{G}$: \[ \Omega_\Gm G \wequi \Gr{G}. \] For a
definition of $\Gr{G}$, see
\cite{zhu2016introduction} or Section \ref{sec:grassmann}. For us, the main
points are as follows: there exists a pointed presheaf of sets $\Gr{G} \in \Pre(\Aff_k)$
(where $\Aff_k$ is the category of all affine $k$-schemes) which
is in fact an fpqc sheaf. Moreover, in the category $\Pre(\Aff_k)$, the sheaf
$\Gr{G}$ is a filtered colimit
\begin{equation} \label{eq:GrG-colim}
  X_1 \to X_2 \to \dots \to \Gr{G},
\end{equation}
where each $X_i$ is (the presheaf represented by) a finite type (but in general
highly singular) $k$-scheme.

\subsection*{Classical analog.} Our result (yet to be stated precisely) has the following classical analog.
Suppose that $k = \C$. Then the complex points $\Gr{G}(\C)$ can be given
the structure of a topological space, namely the colimit of the spaces $X_i(\C)$
(with their strong topology). Then $\Gr{G}(\C)$ is homeomorphic to the
so-called polynomial loop-Grassmannian $\mathrm{Gr}^{G(\C)}_0$
of the Lie group $G(\C)$ \cite[7.2(i)]{pressley1984loop}. This space is
homotopy equivalent to the space of smooth loops $\Omega^{sm}(G(\C)')$, where
$G(\C)'$ is the compact form of $G(\C)$ \cite[Proposition 8.6.6, Theorem 8.6.2]{pressley1984loop},
which itself is well-known to be homotopy equivalent to the usual loop space
$\Omega(G(\C)')$. Finally since $G(\C)' \wequi G(\C)$ (by the
Iwasawa decomposition) we have $\Omega(G(\C)') \wequi \Omega(G(\C))$. Putting
everything together, we have found that
\[ \Gr{G}(\C) \iso \mathrm{Gr}^{G(\C)}_0 \wequi \Omega^{sm}(G(\C)') \wequi
     \Omega(G(\C)') \wequi \Omega(G(\C)). \]

\subsection*{Main result}
In order to state our result precisely, we need to make sense of the ``image of
$\Gr{G}$ in $\Spc(S)_\pt$''. For this we use that the functor $\Sm_k^\aff \subset \Sm_k \to \Spc(k)$
extends, by construction, to a functor $\PSh(\Sm_k^\aff) \to \Spc(k)$, and that we
have a fully faithful inclusion $\Pre(\Sm_k^\aff) \subset \PSh(\Sm_k^\aff)$.
Here $\Sm_k^\aff = \Sm_k \cap \Aff_k$. We thus obtain a functor
\[ \rho: \Pre(\Aff_k) \to \Pre(\Sm_k^\aff) \to \PSh(\Sm_k^\aff) \to \Spc(k), \]
and we also denote by $\rho$ the pointed version $\Pre(\Aff_k)_\pt \to
\Spc(k)_\pt$. This finally allows us to state our main result. For the somewhat
technical notion of isotropic groups, see \cite[Definition
3.3.5]{asok2015affine}. This includes in particular all split groups.

\begin{theorem*}[See Theorem \ref{thm:main}]
Let $k$ be an infinite\footnote{Throughout this note, we make frequent reference
to \cite{asok2015affine}. The main results there are stated only for infinite
fields. However they also apply to finite fields [personal communication], and
an update will appear soon. In this note, whenever we assume that a
field is infinite only because of this reason, we denote this as ``\infinite''.
\label{fn:infinite}} field and $G$ an isotropic reductive $k$-group.
Then we have a canonical equivalence
\[ \Omega_\Gm G \wequi \rho(\Gr{G}) \]
in $\Spc(k)_\pt$.
\end{theorem*}

\subsection*{Organisation and further results}
In Section \ref{sec:Gm-loops} we study the interaction of $Sing_*$ and various
models of $\Omega_\Gm$. Combining this with results of \cite{asok2015affine}, we
obtain a preliminary form of our main computation (see Proposition
\ref{prop:loops-main}): $\Omega_\Gm G$ is motivically
equivalent to the presheaf
\begin{equation} \label{eq:GrG-naive}
  X \mapsto G(X[t, t^{-1}])/G(X).
\end{equation}

In Section
\ref{sec:grassmann}, we review affine Grassmannians. We make no claims to
originality here. The main point is this: $\Gr{G}$ is usually defined as the fpqc
sheafification of the presheaf $X \mapsto G(X\flsr{t})/G(X\fpsr{t})$. We show
that at least over an infinite field, and assuming that $G$ is split, this is
isomorphic to the Zariski sheafification of \eqref{eq:GrG-naive}; see
Proposition \ref{prop:aff-grass-Zar}. We also prove that this is an isomorphism
on sections over \emph{smooth} affine schemes, for any field $k$, and only assuming that $G$
is isotropic; see Proposition \ref{prop:aff-grass-Zar-2}. This is enough for our
eventual application.

In Section \ref{sec:main}, we first deduce the main theorem. This is trivial by
now, since Zariski sheafification is a motivic equivalence. After that we
explore some consequences. We show in Corollary \ref{corr:colim-Ft}
that if $k$ is perfect, then the $\Z[1/e]$-linear motive of $\rho(\Gr{G}) \wequi
\Omega_\Gm G$ is in fact the
filtered colimits of the motives of the singular varieties
$X_i$ from \eqref{eq:GrG-colim}. Since the geometry of the $X_i$ is
well-understood, this allows us in Corollary \ref{corr:motive} to determine the
motive of $\Omega_\Gm G$.

\subsection*{Acknowledgements}
I would like to thank Timo Richarz for patiently explaining many basic facts about affine
Grassmannians, and in particular for explaining to me Lemma \ref{lemm:timo}. I
would also like to thank Maria Yakerson for comments on a draft, and Marc Hoyois
for an enlightening discussion about the Gronthendieck-Serre conjecture and the
ldh topology.

\subsection*{Language and models}
Throughout this note, we frequently switch between various models for motivic
homotopy theory. Section \ref{sec:Gm-loops} is written in the language of
simplicial presheaves and model categories. This is because our manipulations
here are essentially on a point-set level. In particular we employ an
appropriate localisation of the injective local model structure on
$\sPre(\Sm_k^\aff)$ as our model for $\Spc(k)$. Section \ref{sec:grassmann}
deals exclusively with presheaves of \emph{sets}, this time on $\Aff_k$,
reflecting its essentially geometric nature. Finally
Section \ref{sec:main} is written in the language of $\infty$-categories, since
we find our manipulations there are most easily understood in this abstract,
model-independent framework.

\subsection*{Notation}
If $\scr C$ is a small $1$-category, we write $\Pre(\scr C)$ for the
$1$-category of presheaves of sets on $\scr C$, we write $\sPre(\scr C)$ for the
$1$-category of presheaves of simplicial sets on $\scr C$, and we write
$\PSh(\scr C)$ for the $\infty$-category of presheaves of spaces on $\scr C$.

\section{$\Gm$-loops of groups}
\label{sec:Gm-loops}
Let $\scr C$ be an essentially small 1-category with finite products. We write
$\pt \in \scr C$ for the final object. Throughout we fix $\G \in \scr C_{\pt} :=
\scr C_{\pt/}$. We write $\sPre(\scr C)$ for the 1-category of simplicial
presheaves on $\scr C$ and $\sPre(\scr C)_\pt := \sPre(\scr C)_{\pt /}$ for the
pointed version. This admits an \emph{injective model structure} where the weak
equivalences are given objectwise, and the cofibrations are the monomorphisms
\cite[Theorem II.5.8]{jardine-local-homotopy}.
We note that the canonical map $\pt \to \G \in \scr C$ has a section, so is a
monomorphism; in particular $\G \in \sPre(\scr C)_\pt$ is a cofibrant object.

We fix a further object $\A \in \scr C$ together with a map $\G \to \A$.
We call $\scr X \in \sPre(\scr C)$
\emph{$\A$-invariant} if for all $c \in \scr C$, the canonical map $\scr X(c) \to \scr
X(\A \times c)$ is a weak equivalence.

\begin{example}
The case we have in mind is, of course, where $\scr C = \Sm_S$, $\G = \Gm$, and
$\A = \A^1$.
\end{example}

Let us recall that the functor $\sPre(\scr C)_\pt \to \sPre(\scr C)_\pt, \scr X
\mapsto \G \wedge \scr X$ has a right adjoint $\Omega_\G^\naive: \sPre(\scr
C)_\pt \to \sPre(\scr C)_\pt$. It is specified in formulas by asserting that the
following square is cartesian (which in general need \emph{not} imply that it is
homotopy cartesian)
\begin{equation*}
\begin{CD}
  \Omega_\G^\naive(\scr X)(c) @>>> \scr X(\G \times c) \\
       @VVV                             @V{i^*}VV      \\
  \pt                         @>{j_*}>> \scr X(c).
\end{CD}
\end{equation*}
Here $i: \pt \to \G$ is the canonical pointing, as is $j: \pt \to \scr X$. Since
$\G$ is cofibrant, the functor $\Omega_\G^\naive$ is right Quillen (in the
injective model structure we are using), and hence admits a total derived
functor $R\Omega_\G$ which can be computed as $R\Omega_G \scr X \wequi
\Omega_\G^\naive R_f \scr X$, where $R_f$ is a fibrant replacement functor.

\begin{remark}
We denote the underived functor by $\Omega_\G^\naive$ instead of just
$\Omega_\G$ in order to make its point set level nature notationally explicit.
\end{remark}

\begin{remark} \label{rmk:ROmega-computation}
Even if $\scr X$ is objectwise fibrant (i.e. projective fibrant), it need not be
injective fibrant. Indeed a further condition for injective fibrancy is that for
any monomorphism $c \to d \in \scr C$, the induced map $\scr X(d) \to \scr X(c)$
must be a fibration. In particular $\scr X(\G \times c) \to \scr X(c)$ is a
fibration, and we deduce from right properness of the model structure on
simplicial sets \cite[Corollaries II.8.6 and II.8.13]{goerss2009simplicial} that for
any $\scr X \in \sPre(\scr C)_\pt$,
the following diagram is \emph{homotopy cartesian}:
\begin{equation*}
\begin{CD}
  R\Omega_\G(\scr X)(c) @>>> \scr X(\G \times c) \\
       @VVV                             @V{i^*}VV      \\
  \pt                         @>{j_*}>> \scr X(c).
\end{CD}
\end{equation*}

Since $\G$ is not projective cofibrant (in general), the functor
$\Omega_\G^\naive$ is not right Quillen in the projective model structure. In
order to derive it in the projective setting, we first have to cofibrantly
replace $\G$, for example by the cone $\tilde{\G}$ on $* \to \G$. Of course then
$\Omega_{\tilde{\G}} R_f^{proj} \scr X \wequi R\Omega_\G \scr X$.
\end{remark}

Now suppose that $\scr G \in \sPre(\scr C)$ is a presheaf of simplicial groups.
Then $\scr G$ has a canonical pointing, given by the identity section. Thus
$\scr G \in \sPre(\scr C)_\pt$, in a canonical way.
\begin{definition}
We denote by
$\Omega_\G^{gr}(\scr G) \in \sPre(\scr C)$ the simplicial presheaf
\[ \Omega_\G^{gr}(\scr G)(c) = \scr G(\G \times c)/p^* \scr G(c), \]
where $p: \G \times c \to c$ denotes the projection.  
We define a further variant
\[ \Omega_{\G,\A}^{gr}(\scr G)(c) = \scr G(\G \times c)/i^* \scr G(\A \times G), \]
where $i: \G \to \A$ is the canonical map.
\end{definition}
Since $p$ has a section,
$p^*$ is injective and identifies $\scr G(c)$ with a subgroup of $\scr G(\G
\times c)$, so we will drop $p^*$ from the notation. Clearly
$\Omega_\G^{gr}(\scr G), \Omega_{\G,\A}^{gr}(\scr G)$ are
functorial in the presheaf of simplicial groups $\scr G$. Note that unless $\scr G$
is abelian, $\Omega_\G^{gr}(\scr G), \Omega_{\G,\A}^{gr}(\scr G)$ are
not a priori a presheaves of groups.  
Note also that $\scr G(c) \subset \scr
G(\A \times c)$, and hence there is a canonical surjection $\Omega_\G^{gr}(\scr
G) \to \Omega_{\G,\A}^{gr}(\scr G)$.

\begin{definition}
We call $\scr G \in \sPre(\scr C)$ \emph{$(\G,\A)$-injective} if for each $c \in \scr C$,
the restriction $\scr G(\A \times c) \to \scr G(\G \times c)$ is injective.
\end{definition}

\begin{proposition} \label{prop:Omega-gr}
Let $\scr G \in \sPre(\scr C)_\pt$ be a presheaf of simplicial groups,
canonically pointed by the identity.
\begin{enumerate}
\item There is a canonical isomorphism $\Omega_\G^\naive(\scr G) \to
  \Omega_\G^{gr}(\scr G)$.
\item The canonical map $\Omega^\naive_\G(\scr G) \to R\Omega_\G(\scr G)$ is an
  objectwise weak equivalence.
\item Suppose that $\scr G$ is $\A$-invariant and $(\G, \A)$-injective.
  Then the canonical map $\alpha: \Omega_\G^{gr}(\scr G) \to \Omega_{\G,\A}^{gr}(\scr
  G)$ is an objectwise weak equivalence.
\end{enumerate}
\end{proposition}
\begin{proof}
(1) We have for $c \in \scr C$ the canonical map
\[ \alpha_c: \Omega_\G^\naive(\scr G)(c) = ker(\scr G(\G \times c) \to \scr G(c))
                \to \scr G(\G \times c) \to \scr G(\G \times c)/\scr G(c)
                = \Omega_\G^{gr}(\scr G)(c). \]
These fit together to form a canonical map $\Omega_\G^\naive(\scr G) \to
\Omega_\G^{gr}(\scr G)$, which we claim is an isomorphism. Write $j^*: \scr G(\G
\times c) \to \scr G(c)$ for pullback along $\pt \to \G$.
Define a map of sets $\beta: \scr G(\G \times c) \to
\scr G(\G \times c)$ via $\beta(g) = (p^*j^*g)^{-1} g$. If $a \in \scr G(c)$,
then $\beta(ag) = (p^*j^*(ag))^{-1} ag = (a p^*j^*g)^{-1} ag = \beta(g)$.
Furthermore $p^* \beta(g)$ is the identity element of $\scr G(c)$, by
construction. It follows that $\beta$ takes values in $\Omega_\G^\naive(\scr
G)(c)$ and factors through the surjection $\scr G(G \times c) \to
\Omega_\G^{gr}(\scr G)(c)$ to define $\bar{\beta}: \Omega_\G^{gr}(\scr G)(c) \to
\Omega_\G^\naive(\scr G)(c)$. We check immediately that $\bar{\beta}$ is inverse to
$\alpha_c$.

(2) Since $j: \pt \to \G$ has a section ($\pt$ being final), the induced map
$j^*: \scr G(\G \times c) \to \scr G(c)$ is a surjection of simplicial groups,
and hence a fibration \cite[Corollary V.2.7]{goerss2009simplicial}. It follows
from right properness of the model structure on
simplicial sets \cite[Corollary II.8.6]{goerss2009simplicial}
that
\[ \Omega_\G^\naive(\scr G)(c) = fib(\scr G(\G \times c) \to \scr G(c)) \wequi
hofib(\scr G(\G \times c) \to \scr G(c)) \wequi R\Omega_\G(\scr G)(c); \]
see Remark \ref{rmk:ROmega-computation} for the last weak equivalence. Thus the
canonical map is indeed an objectwise weak equivalence.

(3) is an immediate consequence of Lemma \ref{lemm:simplicial-group-subgroup}
below (applied with $G_* = G'_* = \scr G(\G \times c)$, $H_* = \scr G(c), H'_* =
\scr G(\A \times c)$).
\end{proof}

\begin{lemma} \label{lemm:simplicial-group-subgroup}
Let $\theta: G_* \to G'_*$ be a homomorphism of simplicial groups and $H_*
\subset G_*, H'_* \subset G'_*$ simplicial subgroups such that $\theta(H_*)
\subset H'_*$. If each of the maps $\theta: G_* \to G'_*$ and $\theta: H_* \to
H'_*$ are weak equivalences, then so is the induced map $G/H_* \to G'/H'_*$.
\end{lemma}
\begin{proof}
We have $G/H_* \wequi \hocolim_{BH_*} G_*$, since the action is free. Since
the right hand side only depends on $G_*$ and $H_*$ up to weak equivalence, the
result follows.

We can make the above slightly sketchy argument precise as follows. Write
$\bar{G}_*$ for $G_*$ viewed as a bisimplicial set constant in the second
variable, i.e. $\bar{G}_n = G_*$ for all $n$. Define $\bar{H}_*$ similarly. Let $B(H,
G)_*$ be the bisimplicial set $(EH)_* \times \bar{G}_*$,
where $(EH)_n = H_*^{n+1}$. We let $\bar{H}_*$ act on $B(H,G)_*$
diagonally. Then $[B(H,G)/\bar{H}]_n = (H_* \times H_*^n \times G_*)/H_* \iso H_*^n
\times G_*$, and so the canonical
map $B(H,G)/\bar{H}_* \to B(H',G')/\bar{H}'_*$ is a levelwise weak equivalence.
By \cite[Proposition IV.1.7]{goerss2009simplicial}, hence so
is the induced map on diagonals $d(B(H,G)/\bar{H}_*) \to d(B(H',G')/\bar{H}'_*)$.
It is thus enough to prove that
$d(B(H,G)/\bar{H}_*) \wequi G_*/H_*$. The unique map $(EH)_* \to *$ induces $B(H,G) \to
\bar{G}_*$ and then $B(H,G)/\bar{H}_* \to \overline{G/H}_*$. Since
$d(\overline{G/H}_*) = G/H_*$, it
is enough to show that $B(H,G)/\bar{H}_* \to \overline{G/H}_*$ induces a weak equivalence levelwise
in the \emph{other} variable (since taking diagonals is manifestly symmetric in
the two variables). This map is $B(H_n, G_n)/H_n \to G_n/H_n$. It is well-known that
the left hand side is the homotopy orbits of the action of the discrete group
$H_n$ on $G_n$, and the right hand side is the ordinary quotient. They are
weakly equivalent because the action is free.
\end{proof}

To go further, we need to assume that $\A$ is given the structure of a
\emph{representable interval object} \cite[Definition 4.1.1]{asok2015affine-I}.
In this case there is a functor
\[ \Sing_*: \sPre(\scr C)_\pt \to \sPre(\scr C)_\pt \]
with $\Sing_n(\scr X)(c) = \scr X_n(\A^n \times c)$. The functor
$\Sing_*$ preserves objectwise weak equivalences and is in fact a functorial
``$\A$-localization''; in particular it produces $\A$-invariant objects.
All of these properties are mentioned in \cite{asok2015affine-I}, right after
Definition 4.1.4.

\begin{lemma} \label{lemm:Omega-naive-Sing*}
Let $\scr X \in \sPre(\scr C)_\pt$. Then there is a canonical isomorphism
\[ \Omega_\G^\naive \Sing_* \scr X \iso \Sing_* \Omega_\G^\naive \scr X. \]
If $\scr G$ is a presheaf of simplicial groups, then moreover
\[ \Omega_\G^{gr} \Sing_* \scr G \iso \Sing_* \Omega_\G^{gr} \scr G \]
and
\[ \Omega_{\G,\A}^{gr} \Sing_* \scr G \iso \Sing_* \Omega_{\G,\A}^{gr} \scr G \]
\end{lemma}
\begin{proof}
Clear from the defining formulas.
\end{proof}

\begin{corollary} \label{corr:Omega-gm-gr-A1}
Let $\scr G \in \sPre(\scr C)_\pt$ be a presheaf of simplicial groups.
Then
$R\Omega_\G \Sing_* \scr G \wequi \Sing_* \Omega_{\G}^{gr} \scr G$.
If furthermore $\scr G$ is $(\G,\A)$-injective, then
$R\Omega_\G \Sing_* \scr G \wequi \Sing_* \Omega_{\G,\A}^{gr} \scr G$.
\end{corollary}
\begin{proof}
We have
\[ R\Omega_\G \Sing_* \scr G \wequi \Omega_\G^{gr} \Sing_* \scr G
   \iso \Sing_* \Omega_\G^{gr} \scr G, \]
where the first step is by Proposition \ref{prop:Omega-gr}(1,2) and the second
step is by Lemma \ref{lemm:Omega-naive-Sing*}. This proves the first claim.
If $\scr G$ is $(\G, \A)$-invariant, we have furthermore
\[ \Omega_\G^{gr} \Sing_* \scr G \wequi \Omega_{\G,\A}^{gr} \Sing_* \scr G
   \iso \Sing_* \Omega_{\G,\A}^{gr} \scr G, \]
where the first step is by Proposition \ref{prop:Omega-gr}(3), using that $\Sing_*$ produces
$\A$-invariant objects and preserves $(\G,\A)$-injective objects,
and the second step is by Lemma \ref{lemm:Omega-naive-Sing*} again. This proves
the second claim.
\end{proof}

\subsection*{Specialisation to $\A^1$-algebraic topology} We now consider the
situation where $\scr C = \Sm_S^\aff$, $\G = \Gm, \A = \A^1$. Here $S$ is a
Noetherian scheme of finite Krull dimension (in all our applications it will be
the spectrum of a field), and $\Sm_S^\aff$
denotes the category of smooth, finite type, (relative) affine $S$-schemes.

We write $L_\mot \sPre(\Sm_S^\aff)$ for the motivic localization of the model
category $\sPre(\Sm_S^\aff)$ (with its injective global model structure); in
other words the localization inverting $\A^1$-homotopy equivalences and the
distinguished Nisnevich squares (equivalently, the Nisnevich-local weak
equivalences \cite[Lemma 3.1.18]{A1-homotopy-theory}).
It is well-known that $L_\mot
\sPre(\Sm_S^\aff)$ is Quillen equivalent to $L_\mot \sPre(\Sm_S)$, the usual
model for the pointed, unstable motivic homotopy category.\footnote{One way of
seeing this is as follows. The model category $L_\mot
\sPre(\Sm_S)$ is Quillen equivalent to $L_\mot \sShv(\Sm_S)$, where $\sShv$
denotes the category of simplicial Nisnevich-sheaves \cite[Theorem
II.5.9]{jardine-local-homotopy}, and similarly
for $L_\mot \sPre(\Sm_S^\aff)$. But the categories $\sShv(\Sm_S)$ and
$\sShv(\Sm_S^\aff)$ are equivalent, because $\Sm_S$ and $\Sm_S^\aff$ define the
same site.}  We write $L_\mot:
\sPre(\Sm_S^\aff) \to \sPre(\Sm_S^\aff)$ for a fibrant replacement functor for
the motivic model structure.

Let us note right away that $(\G, \A) = (\Gm, \A^1)$-injectivity is common in
our situation.

\begin{lemma} \label{lemm:Gm-A1-injectivity}
If $\scr X \in \sPre(\Sm_S^\aff)$ is represented by a separated $S$-scheme, then
$\scr X$ is $(\Gm,\A^1)$-injective.
\end{lemma}
\begin{proof}
By definition the diagonal $\scr X \to \scr X \times_S \scr X$ is a closed
immersion, whence any two maps $f, g: \A^1 \times U \to \scr X$ over $S$
which agree on $\Gm \times U$ must agree on its closure, which is all of $\A^1
\times U$. In other words, the restriction is injective. This was to be shown.
\end{proof}

We can now state the main result of this section.

\begin{proposition} \label{prop:loops-main}
Let $k$ be an \infinite field and $G$ an isotropic reductive $k$-group. Then
\[ R\Omega_\Gm L_\mot G \wequi \Sing_* \Omega_\Gm^\naive G \iso \Sing_*
    \Omega_\Gm^{gr} G \wequi \Sing_* \Omega_{\Gm,\A^1}^{gr} G, \]
where $\wequi$ means (global) weak equivalence in $\sPre(\Sm_k^\aff)_*$.
\end{proposition}
\begin{proof}
The main point is that under our assumptions, $L_\mot G \wequi \Sing_* G$
\cite[Theorem 4.3.1 and Definiton 2.1.1]{asok2015affine}.
Also $G$ is affine \cite[Definition 3.1.1]{asok2015affine},
so separated, whence $(\Gm, \A^1)$-injective by
Lemma \ref{lemm:Gm-A1-injectivity}. The result now follows from
Lemma \ref{lemm:Omega-naive-Sing*} and Corollary
\ref{corr:Omega-gm-gr-A1}.
\end{proof}

\section{Affine Grassmannians}
\label{sec:grassmann}
We review some basic results about affine Grassmannians. Surely they are all
well-known to workers in the field (i.e., not the author). Our main reference is
\cite{zhu2016introduction}. Throughout, we fix a
field $k$ and write $\Aff_k$ for the category of all affine $k$-schemes (not
necessarily of finite type, not necessarily smooth). We extensively work in the
category $\Pre(\Aff_k)$ of presheaves on affine schemes; as is well-known we
have the Yoneda embedding $\Sch_k \to \Pre(\Aff_k)$. On $\Pre(\Aff_k)$ we have
many topologies, the most relevant for us are the \emph{fpqc} topology
\cite[Tag 03NV]{stacks-project} and the Zariski topology; we denote the relevant sheafification
functors by $a_\fpqc$ (which may not always exist!)
and $a_\Zar$. For elements $\scr F \in \Pre(\Aff_k)$ and
$A$ any $k$-algebra, we put $\scr F(A) := \scr F(Spec(A))$.

\begin{definition}
Let $\scr X \in \Pre(\Aff_k)$ be a presheaf. We have the presheaves $L^+ \scr X,
L \scr X \in \Pre(\Aff_k)$ defined by
\[ L^+ \scr X(A) = \scr X(A\fpsr{t}) \]
and
\[ L \scr X(A) = \scr X(A\flsr{t}). \]
Note that there is a canonical morphism
$L^+ \scr X \to L \scr X$ induced by $A\fpsr{t} \to A\flsr{t}$.

Let $\scr G \in \Pre(\Aff_k)$ be a presheaf of groups. Then $L^+ \scr G, L\scr
G$ are presheaves of groups and we define the \emph{affine Grassmannian} as
\[ \Gr{\scr G} = a_\fpqc L\scr G / L^+ \scr G. \]
\end{definition}

We note right away that at least if $\scr G$ is represented by a group scheme,
then $\Gr{\scr G} = a_\fpqc L\scr G / L^+ \scr G$ exists and is given by $a_\et
L\scr G / L^+ \scr G$ \cite[Proposition 1.3.6, Lemma
1.3.7]{zhu2016introduction}.
Let us further put $L_0 \scr G(A) = \scr G(A[t, t^{-1}])$ and $L_0^+ \scr G(A) =
\scr G(A[t])$. Then we have a commutative square
\begin{equation} \label{eq:L-L+-square}
\begin{CD}
L_0 \scr G @>>> L\scr G \\
@AAA              @AAA  \\
L_0^+ \scr G @>>> L^+ \scr G.
\end{CD}
\end{equation}
The main result of this section is the following. See also Proposition
\ref{prop:aff-grass-Zar-2} at the end of this section for a related and
sometimes stronger result.

\begin{proposition} \label{prop:aff-grass-Zar}
Let $G$ be a split reductive group over an infinite field $k$.
Then the canonical map
\[ a_\Zar L_0 G / L_0^+ G \to \Gr{G} \]
induced by \eqref{eq:L-L+-square} is an isomorphism (of objects in
$\Pre(\Aff_k)$).
\end{proposition}

Before giving the proof, we need some background material. If $\tau$ is a
topology, we call a morphism of
presheaves $f: \scr X \to \scr Y$ a $\tau$-epimorphism if $a_\tau f$ is an
epimorphism in the topos of $\tau$-sheaves.

\begin{definition} \label{def:torsor}
Let $\scr G \in \Pre(\Aff_k)$ be a presheaf of groups acting on $\scr X \in
\Pre(\Aff_k)$. Suppose given a $\scr G$-equivariant map $f: \scr X \to \scr Y$, where
$\scr Y \in \Pre(\Aff_k)$ has the trivial $\scr G$-action. Let $\tau$ be a
topology on $\Aff_k$. We call $f$ a $\tau$-locally trivial $\scr G$-torsor if:
\begin{enumerate}
\item $\scr G, \scr X, \scr Y$ are $\tau$-sheaves,
\item $f$ is a $\tau$-epimorphism, and
\item the canonical map $\scr G \times \scr X \to \scr X \times_{\scr Y} \scr X,
  (g, x) \mapsto (x, gx)$ is an isomorphism.
\end{enumerate}
\end{definition}
Let us note that condition (1) implies that $\scr G \times \scr X$ and $\scr X
\times_{\scr Y} \scr X$ are $\tau$-sheaves, so condition (3) is $\tau$-local.
We call a $\scr G$-torsor \emph{trivial} if there is a $\scr G$-equivariant
isomorphism $\scr X \iso \scr G \times \scr Y$.

\begin{lemma} \label{lemm:torsor-equivalent-definitions}
Suppose that $\scr G$ is a presheaf of groups acting on $\scr X$, and $f: \scr X
\to \scr Y$ is a $\scr G$-equivariant map, where $\scr G$ acts trivially on
$\scr Y$. Suppose that $\scr G, \scr X, \scr Y$ are $\tau$-sheaves.
The following are equivalent.
\begin{enumerate}
\item $f$ is a $\tau$-locally trivial $\scr G$-torsor.
\item For every affine scheme $S$ and every morphism $S \to \scr Y$, there
  exists a $\tau$-cover $\{S_i \to S\}_i$ such that $\scr X_{S_i}$ is a trivial $\scr
  G$-torsor (for every $i$).
\item There exists a $\tau$-epimorphism $\scr U \to \scr Y$ such that
  $\scr X_{\scr U} \to \scr X$ is a trivial $\scr G$-torsor.
\end{enumerate}
\end{lemma}
\begin{proof}
We will work in the topos of $\tau$-sheaves, so suppress any mention of
$\tau$-sheafification, and also say ``epimorphism'' instead of
``$\tau$-epimorphisms'', and so on.

(1) $\Rightarrow$ (2): Let $S \to \scr Y$ be any map. Since epimorphisms in a
topos are stable under base change (e.g. by universality of colimits),
$\alpha: \scr X_S \to S$ is also a
$\scr G$-torsor, and in particular an epimorphism. There exists then a cover
$\{S_i \to S\}_i$ over which $\alpha$ has a section, being an epimorphism. In other
words, $\scr X_{S_i} \to S'$ is trivial, as required.

(2) $\Rightarrow$ (3): Taking the coproduct $\coprod_{S \to \scr Y} \coprod_i
S_i \to \scr Y$
over a sufficiently large collection of affine schemes $S$ mapping to $\scr Y$,
we obtain a trivializing epimorphism as required.

(3) $\Rightarrow$ (1): We need to prove that $\scr X \to \scr Y$ is an
epimorphism and that $\scr G \times \scr X \to \scr X \times_{\scr Y} \scr X$ is
an isomorphism. Both statements may be checked after pullback along the
(effective) epimorphism $\scr U \to \scr Y$. We may thus assume that $\scr X \to
\scr Y$ is trivial, in which case the result is clear.
\end{proof}

\begin{lemma} \label{lemm:torsor-quotient}
Let $\scr X \to \scr Y$ be a $\tau$-locally trivial $\scr G$-torsor.
Then $\scr Y \iso a_\tau \scr X/\scr G$.
\end{lemma}
\begin{proof}
We again work in the topos of $\tau$-sheaves.
By definition $\scr X \to \scr Y$ is an epimorphism. Since every epimorphism in a
topos is effective \cite[Tag 086K]{stacks-project},
we have a coequaliser $\scr X \times_{\scr Y} \scr X \rightrightarrows \scr X
\to \scr Y$ in $\tau$-sheaves.
By condition (3) of Definition \ref{def:torsor}, this is the action coequalizer. The result follows.
\end{proof}

\begin{proof}[Proof of Proposition \ref{prop:aff-grass-Zar}.]
By Lemma \ref{lemm:torsor-quotient}, it suffices to prove that $L_0 G \to
\Gr{G}$ is a Zariski-locally trivial $L_0^+ G$-torsor. All presheaves involved
are fpqc-, and hence Zariski-sheaves.

We shall make use of results from
\cite[Section 2]{zhu2016introduction}. There $k$ is assumed to be algebraically
closed. This will not matter in each case we cite this reference, because the
property we are checking will be fpqc local.

We introduce some additional notation. We denote by $L^- G$ the presheaf $A
\mapsto G(A[t^{-1}])$. We have $L^-G \iso L^+_0 G$, but the canonical embedding
$L^- G \to L_0 G$ is different. Evaluation at $t^{-1} = 0$ induces $L^- G \to G$, and we
let $L^{<0} G = ker(L^- G \to G)$. I claim that the following square is a
pullback, where $i$ is the multipication map
\begin{equation*}
\begin{CD}
L^{<0} G \times L_0^+G @>i>> L_0 G \\
@V{pr_1}VV                     @VVV  \\
L^{<0}G             @>j>> \Gr{G}.
\end{CD}
\end{equation*}
In order to see this, we note first that it follows from \cite[Lemma
3.1]{richarz-test-function} that $i$ is a monomorphism.
Let $\scr F = L^{<0}G \times_{\Gr{G}} L_0 G$. Since $i$ is a mono so is the
canonical map $\alpha: L^{<0}G \times L_0^+G \to \scr F$. Let
$x, y \in \scr F(A)$, corresponding to $x \in L^{<0}G(A)$ and $y \in L_0G(A)$
with the same image in $\Gr{G}(A)$. In other words, fpqc-locally on $A$ we can
find $z \in L_0^+ G(A)$ with $y = xz$. Thus $\alpha$ is fpqc-locally
an epimorphism. It is thus an fpqc-local isomorphism of fpqc-sheaves, and hence
an isomorphism (of presheaves).

Now let $A \in L_0G(k)$. We obtain a map $j_A: L^{<0}G \to \Gr{G}, x \mapsto
A \cdot j(x)$. Define similarly $i_A: L^{<0} G \times L_0^+G \to L_0 G, (x, y)
\mapsto Axy$. Since $A$ is invertible, the following square is also a pullback
\begin{equation*}
\begin{CD}
L^{<0} G \times L_0^+G @>i_A>> L_0 G \\
@V{pr_1}VV                     @VVV  \\
L^{<0}G             @>j_A>> \Gr{G}.
\end{CD}
\end{equation*}
By Lemma \ref{lemm:torsor-equivalent-definitions}, it is thus enough to show
that $j' = \coprod_{A \in L_0 G(k)} j_A$ is a Zariski-epimorphism. Note first that
$j_A$ is a morphism of ind-schemes \cite[Theorem 1.2.2]{zhu2016introduction},
and an open immersion \cite[Lemma 3.1]{richarz-test-function}.
Consequently
each $j_A$ identifies an open ind-subscheme. In order to check that $j'$ is a
Zariski-epimorphism, it suffices to check that the $j_A$ form a covering. Let
$\bar{k}$ denote an algebraic closure of $k$. Since
$\Gr{G}$ is of ind-finite type, it suffices to check this on $\bar{k}$-points.
The result thus follows from Lemma \ref{lemm:timo} below.
\end{proof}

The above proof is complete if $k = \bar{k}$. In the general case, we need the
following result, which is probably well-known to experts. A proof was kindly
communicated by Timo Richarz.

\begin{lemma}\label{lemm:timo}
Let $k$ be an infinite field, $\bar{k}$ an algebraic closure, and $G$ a split
reductive group over $k$. Then $\Gr{G}(\bar{k})$ is covered by the translates $A
L^{<0} G(\bar{k}) \subset \Gr{G}(\bar{k})$, for $A \in L_0 G(k)$.
\end{lemma}
\begin{proof}
We shall make use of the \emph{Bruhat decomposition} of $\Gr{G}$. Namely, there
exists a set $X$, together with for each $\mu \in X$ an element $t^\mu \in
L_0G(k)$ and a $k$-scheme $U_\mu \subset L_0G$ such that
\begin{enumerate}
\item The canonical map $U_\mu \to \Gr{G}, A \mapsto A t^\mu \cdot e$ is a
  locally closed embedding. Denote the image by $Y_\mu$.
\item There is an isomorphism $U_\mu \iso \mathbb{A}^{l(\mu)}$ for some
  non-negative integer $l(\mu)$.
\item The schemes $Y_\mu \to \Gr{G}$ form a locally closed cover.
\end{enumerate}
We do not know a good reference for the statement in this generality,
but see for example \cite[Theorem 8.6.3]{pressley1984loop}.

It is clear that $L_0 G \to \Gr{G}$ is trivial over $Y_\mu$. We deduce that (1) $L_0G \to
\Gr{G}$ is surjective on $k$-points. Put $\bar{G} = G_{\bar{k}}$.
We claim that (2) any non-empty open
$L_0^+\bar{G}$-orbit in $L_0\bar{G}$
contains (the image of) a $k$-point (of $L_0 G$). 
Using surjectivity on $k$-points, for this it
suffices to prove that any non-empty open $U \subset \Gr{\bar{G}}$ contains a
$k$-rational
point. Being non-empty, $U$ meets $\bar{Y}_\mu := (Y_\mu)_{\bar k}$
for some $\mu$. Then $\bar{V} := \bar{Y}_\mu \cap U$ is a non-empty open subset
of $\bar{Y}_\mu \iso \A^n_{\bar k}$ for some $n$. Its image $V$ in $\A^n_k$ is open
\cite[Tag 0383]{stacks-project} and non-empty. Since $k$ is infinite, $V$ has a
rational point\footnote{This result is widely known and easy to prove, yet we
could not locate a reference. A proof is recorded on MathOverflow at
\cite{264212}.}.
This establishes the claim.

Finally let $\bar{A} \in \Gr{G}(\bar{k})$. By surjectivity on
$\bar{k}$-points (1), we find $A \in L_0G(\bar{k})$ mapping to $\bar{A}$.
Consider the $L^-\bar{G}$-orbit $O = A L_0^+ \bar{G} L^-\bar{G} \subset
L_0 \bar{G}$. I claim that $O$ contains a $k$-point.
The automorphism $\mathrm{rev}: L_0 G \to L_0 G$ induced by $t
\mapsto t^{-1}$ interchanges $L^-$ and $L_0^+$, and hence converts $L^-$
orbits into $L_0^+$-orbits. Since it is defined over $k$ it
preserves $k$-points. It is hence enough to show $\mathrm{rev}(O)$ has a
$k$-point, and by the claim (2) it is enough to show that $\mathrm{rev}(O)$ is
open.
But $\mathrm{rev}(O) = \mathrm{rev}(A) L^-\bar{G} L_0^+ \bar{G}$ which is open, being the preimage of
$\mathrm{rev}(A) L^{<0} \bar{G} \subset \Gr{\bar{G}}$.

We thus find $B \in L_0^+G(\bar{k}), C \in
L^-G(\bar{k})$ such that $ABC \in L_0G(k)$. Then
\begin{gather*}
  \bar{A} = A \cdot e = (ABC) C^{-1} B^{-1} \cdot e = (ABC) C^{-1} \cdot e \\ 
  \in (ABC) L^- G(\bar{k}) \cdot e = (ABC) L^{<0} G(\bar{k}) \cdot e \subset \Gr{G}(\bar{k}).
\end{gather*}
This was to be shown.
\end{proof}

\begin{remark}
There is an alternative proof of Proposition \ref{prop:aff-grass-Zar}, using a
recent result of Fedorov \cite{fedorov2015grothendieck}. Moreover this proof
does not require $k$ to be infinite, or a field. It was also kindly communicated
by Timo Richarz.
\end{remark}
\begin{proof}[Alternative proof of Proposition \ref{prop:aff-grass-Zar}.]
It follows from the Beauville-Laszlo gluing lemma \cite{beauville1995lemme} that
$\Gr{G} \iso a_{fpqc} L_0 G / L_0^+ G \iso T$, where $T$ is the functor sending
$Spec(A)$ to the set of isomorphism classes of tuples $(\scr F, \alpha)$ with
$\scr F$ a $G$-torsor on $\A^1_A$ and $\alpha$ a trivialization of $\scr F$ over
$\A^1_A \setminus \{0\}$. The map $L_0 G \to T$ sends $M \in L_0 G(A)$ to the
pair $(\scr F^0, \alpha_M)$, where $\scr F^0$ is the trivial $G$-torsor and
$\alpha_M$ is the trivialization induced by $M$. By Lemmas
\ref{lemm:torsor-equivalent-definitions} and \ref{lemm:torsor-quotient}, what we
need to show is that this map $L_0 G \to T$ admits sections Zariski-locally on
$T$. In other words if $Spec(A) \in \Aff_k$ and $(\scr F, \alpha) \in T(A)$,
then the $G$-torsor $\scr F$ over $\A^1_A$ is Zariski-locally on $A$ trivial.

If
$A$ is Noetherian local, this is \cite[Theorem 2]{fedorov2015grothendieck}.
We
need to extend this to more general $A$, so let $Spec(A) \in \Aff_k$ and $(\scr
F, \alpha) \in T(A)$ be
arbitrary. We may write $A$ as a filtering colimit of Noetherian rings $A_i$.
Since $\Gr{G}$ is of ind-finite type, we find $(\scr F', \alpha') \in T(A_i)$
for some $i$ inducing $(\scr F, \alpha)$. Thus we may assume that $A$ is
Noetherian. Fedorov's result assures us that $\scr F$ is trivial over any local
ring of $A$. Thus what remains to show is that triviality of $\scr F$ on
$\A^1_A$ (or equivalently $\P^1_A$) is an open condition on $Spec(A)$. This is proved
in \cite[proof of Lemma 3.1]{richarz-test-function}.
\end{proof}

We can also prove the following related result, tailored to our narrower
applications.

\begin{proposition} \label{prop:aff-grass-Zar-2}
Let $G$ be an isotropic reductive group over an \infinite field $k$.
Then the canonical map
\[ a_\Zar L_0 G / L_0^+ G \to \Gr{G} \]
induced by \eqref{eq:L-L+-square} is an isomorphism on sections over smooth
affine varieties.
\end{proposition}
\begin{proof}
By arguing as in the alternative proof of Proposition \ref{prop:aff-grass-Zar},
what we need to show is the following: if $X$ is a smooth affine variety and
$\scr F$ is a $G$-torsor on $\A^1_X$ which is trivial over $\A^1_X \setminus
\{0\}$, then $\scr F$ is Zariski-locally on $X$ trivial. By definition $\scr F$ is
generically trivial, and hence by the resolution of the Grothendieck-Serre
conjecture over fields \cite{fedorov2015proof, panin2017proof}, $\scr F$ is
Zariski-locally trivial (on $\A^1_X$). By homotopy invariance for
$G$-torsors over smooth affine schemes \cite[Theorem 3.3.7]{asok2015affine},
we find that $\scr F \iso (\A^1_X \to X)^* \scr G$,
for some Nisnevich-locally trivial $G$-torsor $\scr G$ on $X$.
Now $\scr G$ is generically trivial, so by Grothendieck-Serre again $\scr
G$ is Zariski-locally trivial. This concludes the proof.
\end{proof}

\section{Main result}
\label{sec:main}

We now come to our main result. Let $\Spc(k)_\pt$ denote the $\infty$-category of
pointed motivic spaces. As usual we have a canonical functor $(\Sm_k)_\pt \to
\Spc(k)_\pt$. We also have the functor $\rho: \Pre(\Aff_k)_\pt \to \Spc(k)_\pt$.
It is obtained as the composite
\[ \Pre(\Aff_k)_\pt \xrightarrow{j^*} \Pre(\Sm^\aff_k)_\pt
     \xrightarrow{L} \Spc(k)_\pt, \]
where the $j^*$ is restriction along $j: \Sm^\aff_k \to \Aff_k$ and $L$ is the
motivic localization functor. Recall also the $\Gm$-loops functor $R^{\A^1}\Omega_\Gm:
\Spc(k)_\pt \to \Spc(k)_\pt$.

\begin{theorem} \label{thm:main}
Let $k$ be an \infinite field and $G$ an isotropic reductive $k$-group.
Then we have a canonical equivalence
\[ R^{\A^1} \Omega_\Gm G \wequi \rho(\Gr{G}) \]
in $\Spc(k)_\pt$. Here $\Gr{G}$ is pointed by the image of the identity element
in $G$.
\end{theorem}
\begin{proof}
By Proposition \ref{prop:loops-main}, we have $R^{\A^1} \Omega_\Gm G =
R\Omega_\Gm L_{mot} G \wequi
\Omega_{\Gm,\A^1}^{gr} G$, a weak equivalence in $\Spc(k)_\pt$. In the notation
of Section \ref{sec:grassmann}, we have $\Omega_{\Gm,\A^1}^{gr} G = j^*(L_0
G/L_0^+ G)$. For $F \in \Pre(\Sm_S^\aff)_\pt$ the map $F \to a_\Zar F$ is a motivic
equivalence, i.e. becomes an equivalence in $\Spc(k)_\pt$.
Since $j^*$ commutes with $a_\Zar$, the result now follows from Proposition
\ref{prop:aff-grass-Zar-2}.
\end{proof}

\begin{example}
Group schemes $G$ satisfying the assumptions of Theorem \ref{thm:main} are,
among many others, $\mathrm{GL}_n, \mathrm{SL}_n, \mathrm{Sp}_n$.
\end{example}

\subsection*{Motives of singular varieties}
The presheaves $\Gr{G}$ are well-understood: they are filtered colimits of
projective varieties. Unfortunately, these projective varieties are highly
singular. Thus we need to incorporate motives of singular varieties in
order to make the best use of Theorem \ref{thm:main}.

Let $\Ft_k$ denote the category of finite type $k$-schemes, and suppose that $k$
has exponential characteristic $e$ (i.e. $e=1$ if $char(k) = 0$ and $e=p$ if
$char(k) = p > 0$). Recall the
$\infty$-category $\DM(k, \Z[1/e])$ of $\Z[1/e]$-linear motives
\cite{voevodsky-triang-motives} and the functor $M: \Spc(k)_\pt \to \DM(k, \Z)
\to \DM(k, \Z[1/e])$ sending a pointed motivic space to its motive.
There also is a more complicated functor
\[ \ul{M}: \Pre(\Ft_k)_\pt \to \DM(k, \Z[1/e]); \]
we recall its definition below in the proof of Proposition
\ref{prop:singularities}.
For $X \in (\Sm_k)_\pt$ we have $M X \wequi \ul{M} X$,
where on the right hand side we view $X$ as an element of $(\Ft_k)_\pt
\subset \Pre(\Ft_k)_\pt$. In other words, the functor $\ul{M}$ allows us to
make sense the motive of (among other things) singular varieties.

Denote by $e^*: \Pre(\Ft_k)_\pt \to \Pre(\Sm_k^\aff)_\pt$ the functor of
restriction along the canonical inclusion $\Sm_k^\aff \to \Ft_k$.

\begin{proposition} \label{prop:singularities}
Let $k$ be a perfect field and $\scr X \in \Pre(\Ft_k)_\pt$.
Then $M e^* \scr X \wequi \ul{M} \scr X$.
\end{proposition}
\begin{proof}
For a small (1-)category $\scr C$, denote by $\PSh(\scr C)$ the
$\infty$-category of (space-valued) presheaves on $\scr C$. If $f: \scr C \to
\scr D$ is a functor, we get an adjunction $f: \PSh(\scr C) \adj \PSh(\scr D):
f^*$. We have a full inclusion $\Pre(\scr C) \subset \PSh(\scr C)$ and similarly
for $\scr D$, and the
following diagram commutes:
\begin{equation*}
\begin{CD}
\Pre(\scr C) @>>> \PSh(\scr C) \\
@Af^*AA              @Af^*AA   \\
\Pre(\scr D) @>>> \PSh(\scr D).
\end{CD}
\end{equation*}

The functor $\ul{M}$ is constructed via the following commutative diagram
\begin{equation*}
\begin{CD}
\PSh(\Ft_k)_\pt @>{\mu}>> \ul{\DM}(k, \Z[1/e])  \\
@AeAA                           @Ae^MAA     \\
\PSh(\Sm_k^\aff)_\pt @>{M}>> \DM(k, \Z[1/e]).
\end{CD}
\end{equation*}
The category $\ul{\DM}(k, \Z[1/e])$ can be defined as the $T$-stabilisation
of $L_{cdh, \A^1} \PSh_\Sigma(Cor(k))$, where $Cor(k)$ is the category of
finite correspondences \cite[Theorem 2.5.9]{kelly2013triangulated} and
$\PSh_\Sigma$ denotes the nonabelian derived category \cite[Section
5.5.8]{lurie-htt}.
The functor $e^{M}$ is the stabilisation of the derived left Kan extension functor $e:
\PSh_\Sigma(SmCor(k)) \to \PSh_\Sigma(Cor(k))$. The important result is that
$e^{M}$ is an equivalence \cite[Corollary 5.3.9]{kelly2013triangulated};
one puts $\ul{M} = (e^{M})^{-1} \circ \mu$.

In order to prove our result, it is thus enough to show that the co-unit map
$\eta: e e^* \scr X \to \scr X$ is inverted by the functor $\mu$. For this it
suffices to show that the image $\mu_l(\eta) \in \ul{\DM}(k, \Z_{(l)})$ of $\mu(\eta)$ is
an equivalence for all primes $l \ne p$. The functor $\mu_l$ inverts local
equivalences for the so-called $ldh$-topology \cite[Corollary 5.3.9]{kelly2013triangulated},
and all finite type $k$-schemes are $ldh$-locally smooth \cite[Corollary
3.2.13]{kelly2013triangulated}. It is hence enough to show that $e^*(\eta): e^*
e e^* \scr X \to e^* \scr X$ is an equivalence. This follows from the fact that
$e^* e \wequi \id_{\PSh(\Sm_k^\aff)_\pt}$, which itself is a consequence of
fully faithfulness of $e: \Sm_k^\aff \to \Ft_k$.
\end{proof}

\begin{remark}
If $k$ has characteristic $0$, then using \cite{voevodsky2008unstable}
for $\scr X \in \Pre(\Ft_k)_\pt$ one may
define the $S^1$-stable homotopy type $\ul{\Sigma}^\infty_s \scr X \in \SH^{S^1}(k)$.
Essentially the same proof as above shows that $\Sigma^\infty_s e^* \scr X \wequi
\ul{\Sigma}^\infty_s \scr X \in \SH^{S^1}(k)$.
In positive characteristic, there does not seem to
be an equally accessible $S^1$-stable homotopy type of singular varieties.
\end{remark}

\begin{corollary} \label{corr:colim-Ft}
Let $X_1 \to X_2 \to \dots \in (\Ft_k)_\pt$ be a directed system of pointed
finite type $k$-schemes. View each $X_i$ as an element of $\Pre(\Aff_k)_\pt$ and
put $\scr X = \colim_i X_i \in \Pre(\Aff_k)_\pt$. Then we have $M \rho(\scr X)
\wequi \colim_i \ul{M} X_i$.
\end{corollary}
We note that a filtered colimit of fpqc-sheaves (computed in $\Pre(\Aff_k)$) is an fpqc-sheaf
\cite[Tags 0738 (4) and 022E]{stacks-project},
and so the colimit in the corollary can be computed in the category of fpqc-sheaves.
\begin{proof}
Let $\scr X' \in \Pre(\Ft_k)_\pt$ be the colimit viewed as a presheaf on finite
type schemes. Then $e^*(\scr X') = j^*(\scr X)$ and the result follows from
Proposition \ref{prop:singularities}, using that all functors in sight commute
with filtered colimits.
\end{proof}

\begin{corollary} \label{corr:motive}
Let $G$ be an isotropic reductive group over an \infinite perfect field $k$ of
exponential characteristic $e$.
Then we have
\[ M(\rho(\Gr{G})) \wequi \bigoplus_{\mu \in X(G)} \Z[1/e](l(\mu))[2l(\mu)]
                                         \in \DM(k, \Z[1/e]). \]
Here $X(G)$ is the set of cocharacters of $G$, and $l(\mu)$ is the
(non-negative) integer from the proof of Lemma \ref{lemm:timo}.
\end{corollary}
\begin{proof}
The Bruhat decomposition provides a filtration of $\Gr{G}$ by closed
subschemes $\emptyset = X_{-1} \subset X_0 \subset X_1 \subset \dots \subset
\Gr{G}$. Here $X_i
= \cup_{l(\mu) \le i} Y_\mu$, in the notation of the proof of Lemma \ref{lemm:timo}.
Then $X_{i-1} \to X_i$ is a closed immersion with open complement isomorphic to
$\coprod_{l(\mu) = i} \A^i$. It is well-known that this implies our result. For
the convenience of the reader, we review the standard argument.
We have $M(\rho(\Gr{G})) = \colim_i M(X_i)$ (by Corollary \ref{corr:colim-Ft}), and $M(X_{-1}) =
0$, so it suffices to prove that $M(X_i) = M(X_{i-1}) \oplus \bigoplus_{l(\mu)=i}
\Z\{i\}[2i]$. Since each $X_i$ is projective, we have $M(X_i) = M^c(X_i)$, where
$M^c$ denotes the motive with compact support \cite[p. 9]{voevodsky-triang-motives}. Thus we can use the
Gysin triangle with compact support \cite[Proposition 4.1.5]{voevodsky-triang-motives}
\[ M^c(X_{i-1}) \to M^c(X_{i}) \to M^c(X_{i} \setminus X_{i-1})
    \to M^c(X_{i-1})[1]. \]
Since $M^c(\A^i) = \Z(i)[2i]$, the boundary map vanishes for weight reasons (by
induction, $M^c(X_{i-1})$ is a sum of Tate motives of weight $<i$),
giving the desired splitting.
\end{proof}

\bibliographystyle{plainc}
\bibliography{bibliography}

\end{document}